\documentclass{amsart}
\usepackage{amssymb,latexsym}
\usepackage{amsmath}
\usepackage{graphicx}
\usepackage{amscd}
\usepackage{color}
\usepackage{enumerate}
\newenvironment{enumeratei}{\begin{enumerate}[\upshape (i)]}{\end{enumerate}}
\numberwithin{equation}{section}
\theoremstyle{plain}
 \newtheorem{theorem}{Theorem}[section]
 \newtheorem{lemma}[theorem]{Lemma}
 \newtheorem{proposition}[theorem]{Proposition}
 
 \newtheorem{corollary}[theorem]{Corollary}
\theoremstyle{definition}

\theoremstyle{remark}

\newcommand\idlat [1] {\mathfrak I(#1)}
\newcommand\auv[3] {a^{#1#2}_{#3}}
\newcommand\buv[3] {b^{#1#2}_{#3}}
\newcommand\cfel[3] {\widehat c^{\kern 1.5pt#1#2}_{#3}}
\newcommand\cdn[3] {\breve c^{\kern 0.5pt#1#2}_{#3}}
\newcommand\dfel[3] {\widehat d^{\kern 1.5pt#1#2}_{#3}}
\newcommand\ddn[3] {\breve d^{\kern 0.5pt#1#2}_{#3}}
\newcommand \xx {\xi}
\newcommand \yy {\psi}
\newcommand \zz {\zeta}

\newcommand \Equ {\textup{Equ}}

\newcommand \rank {r}
\newcommand \hal {\widehat\alpha}
\newcommand \hbe {\widehat\beta}
\newcommand \hga {\widehat\gamma}
\newcommand \hde {\widehat\delta}
\newcommand \hep {\widehat\epsilon}
\newcommand \htt [1]{\widehat t^{\kern 0.8pt(#1)}}

\newcommand \szi[1] {\sigma^{(#1)}}
\newcommand \Suv [2]{S^{#1#2}}
\newcommand \tuple [1] {\langle #1\rangle}
\newcommand \pair [2] {\tuple{#1,#2}}
\newcommand \blokk[2] {#1/#2}
\newcommand \restrict[2] {{#1\rceil_{#2}}}
\newcommand \ideal {\mathord{\downarrow}}
\renewcommand\emptyset{\varnothing}
\newcommand\red[1]{{\textcolor{red}{#1}}}
\newcommand \tbf [1] {\textbf{#1}} 
\newcommand \set[1] {\{#1\}}

\newcommand \zero[1]{\Delta_{#1}}
\newcommand \unit[1]{\nabla_{\kern -3pt #1}}
\newcommand \carm {\mathfrak m}
\newcommand \fcitind [2] {{^{\textup{f}}\kern-1pt #1_{\textup{\cite{#2}}}}}
\newcommand \icitind [2] {{^{\textup{i}}\kern-1pt #1_{\textup{\cite{#2}}}}}

\begin{document}
\title[Sublattices of three-generated lattices] 
{Lattices  embeddable in  three-generated lattices}

\author[G.\ Cz\'edli]{G\'abor Cz\'edli}
\email{czedli@math.u-szeged.hu}
\urladdr{http://www.math.u-szeged.hu/\textasciitilde{}czedli/}
\address{University of Szeged, Bolyai Institute, 
Szeged, Aradi v\'ertan\'uk tere 1, HUNGARY 6720}

\thanks{This research was supported by
NFSR of Hungary (OTKA), grant number K 115518}

\dedicatory{Dedicated to the memory of professor L\'aszl\'o Megyesi $(\kern-1pt$1939--2015$\kern1pt)$, former head of the Department of Algebra and Number Theory of the University of Szeged}

\date{\hfill \red{December 12, 2015}}



\begin{abstract} 
We prove that every finite lattice $L$ can be embedded in a three-generated \emph{finite} lattice $K$. We also prove that every \emph{algebraic} lattice with accessible cardinality is a \emph{complete} sublattice of an appropriate \emph{algebraic} lattice $K$ such that $K$ is completely generated by three elements. Note that ZFC has a model in which all cardinal numbers are accessible. 
Our results strengthen P.\ Crawley and R.\ A.\ Dean's 1959 results by adding finiteness, algebraicity, and completeness.
\end{abstract}

\maketitle

\section{Introduction and result}
A lattice $K$ is \emph{three-generated} if there are $a,b,c\in K$ such that for every sublattice $S$ of $K$, $\set{a,b,c}\subseteq S$ implies $S=K$. 
We know from Crawley and Dean~\cite[Theorem 7]{crawleydean} that every at most countably infinite lattice $L$ is a sublattice of a three-generated lattice. A complete lattice  $L$ is \emph{completely} generated by a subset $X$ if the only complete sublattice of $L$ including $X$ is $L$ itself. By
\cite[Theorem 7]{crawleydean} again, every lattice is a sublattice of a complete lattice completely generated by three elements. Our aim is to strengthen these statements for lattices $L$ such that $|L|$ is neither $\aleph_0$, nor rather large; see Corollaries~\ref{corolone} and \ref{coroltwo}  soon. 
Actually, we are going to prove a theorem on equivalence lattices; our theorem combined with deep results from the literature will easily imply the corollaries.

Next, we recall some well-known definitions. 
An element $a$ in a complete lattice $L$ is \emph{compact} if whenever $a\leq \bigvee X$, then $X$ has a finite subset $Y$ with $a\leq \bigvee Y$. A complete lattice is \emph{algebraic} if each of its elements is the join of (possibly, infinitely many) compact elements. Most lattices related to algebraic and other mathematical structures are algebraic lattices.
A cardinal $\kappa$ is \emph{inaccessible} if 
\begin{enumeratei}
\item $\kappa>\aleph_0$,
\item for every cardinal $\lambda$, $\,\lambda<\kappa$ implies $2^\lambda<\kappa$, and 
\item for every set $I$ of cardinals, if $|I|<\kappa$ and each member of $I$ is less than $\kappa$, then $\sum\set{\lambda: \lambda\in I}<\kappa$. 
\end{enumeratei}
For convenience, a cardinal $\lambda$ will be called \emph{accessible} if there is no inaccessible cardinal $\kappa$ such that $\kappa\leq \lambda$. Our terminology is slightly different from the one used in Keisler and Tarski~\cite{keislertarski}, since finite cardinals are accessible here.
Note that there can be cardinals that are neither accessible, nor inaccessible. 
Inaccessible cardinals are also called strongly inaccessible. Inaccessible cardinals, if exist, are extremely large; see, for example,  Kanamori~\cite[page 18]{kanamori} or  Levy~\cite[pages 138--141]{levy}. 
The ``everyday's cardinals'' like  $0,1,2,3$,\dots, $\aleph_0$, $\aleph_1$, $\aleph_2$, $\aleph_3$, $2^{\aleph_0}$, $2^{2^{\aleph_0}}$, etc.~ are accessible.     We know from Kuratowski~\cite{kuratowski}, see also \cite[page 18]{kanamori} or \cite{levy}, that ZFC has a model without inaccessible cardinals. 
That is, in some model of ZFC,  all  cardinals are accessible and  belong to the scope of our results.
Given a set $A$, the lattice of all equivalence relations on $A$ with respect to set inclusion is 
denoted by $\Equ(A)=\tuple{\Equ(A); \subseteq}$; 
it is called the \emph{equivalence lattice} over $A$. Now, we are in the position to formulate the main result of the paper. 

\begin{theorem}\label{thmmain} For every  set $A$ of accessible cardinality $|A|\geq 3$, there exist a set $B$ and a complete sublattice $K$ of $\Equ(B)$ such that
$K$ is completely generated by three of its elements and $\Equ(A)$ is isomorphic to a complete sublattice of $K$. Furthermore, we can choose a finite $B$ if $A$ is finite, and we can let $B=A$~otherwise. 
\end{theorem}

The role of $|A|\geq 3$ in Theorem~\ref{thmmain} is to ensure that $|\Equ(A)|\geq 3$.
Note that a finite lattice is completely generated by three elements if and only if it is  three-generated in the usual sense. 
Note also that $K$ in the theorem is necessarily a proper complete sublattice of $\Equ(B)$ if $|B|>3$, because $\Equ(B)$ cannot be completely generated by three elements; see Strietz~\cite{strietz} or Z\'adori~\cite{zadori} for the finite case,  and see the paragraph following the theorem in \cite{czgfourlarge} for the infinite case.  
For the sake of another terminology, we rephrase Theorem~\ref{thmmain} as follows. A \emph{complete embedding} is an injective map preserving arbitrary (possibly infinite) meets and joins.

\begin{proposition} If $A$ is a set and $|A|\geq 3$ is an  accessible cardinal, then $\Equ(A)$ has a complete embedding in 
a complete sublattice $K$ of some $\Equ(B)$ such that $K$ is completely generated by only three elements. We can let $B=A$ if $A$ is infinite, and we can choose a finite $B$ if $A$ is finite.
\end{proposition}

\begin{corollary}\label{corolone}
Every finite lattice is the sublattice of an appropriate three-generated \emph{finite} lattice. 
\end{corollary}

To point out the difference between  this corollary and the afore-mentioned Crawley and Dean~\cite[Theorem 7]{crawleydean}, note that Corollary~\ref{corolone} refers to a construction that preserves finiteness. However, as opposed to \cite[Theorem 7]{crawleydean}, our corollary does not include the case when the original lattice is countably infinite.

\begin{proof}[{Proof of Corollary~\ref{corolone}}] By Pudl\'ak and T\r uma~\cite{pudlaktuma}, we can embed our lattice in some  $\Equ(A)$ such that $A$ is finite.  Thus,  Theorem~\ref{thmmain} is applicable.
\end{proof}

\begin{corollary}\label{coroltwo}
Every algebraic lattice of accessible cardinality is the \emph{complete} sublattice of an \emph{algebraic} lattice $K$ that is completely generated by three elements.
\end{corollary}

Theorem 7 for complete lattices in  Crawley and Dean~\cite{crawleydean}  pays attention to the ``degree'' of completeness, so it is more involved than its simplified form given at the beginning of this section. However, Corollary~\ref{coroltwo} adds two new features even to the original 
\cite[Theorem 7]{crawleydean}: here the sublattice is a \emph{complete} sublattice and $K$ is an \emph{algebraic} lattice. 
Note that we can easily derive the simplified form of \cite[Theorem~7]{crawleydean} for a lattice $L$ with  accessible cardinality from Corollary~\ref{coroltwo} as follows:  embed $L$  in its \emph{ideal lattice} $\idlat L$ and apply Corollary~\ref{coroltwo} to $\idlat L$; in this way, we obtain that $L$ is (isomorphic to) a sublattice of $K$, where $K$ is 
from Corollary~\ref{coroltwo}.

Conversely, we do not see any straightforward way to derive Corollary~\ref{coroltwo} directly from 
 \cite[Theorem~7]{crawleydean}. 
In particular, if $K$ is a complete lattice completely generated by $\set{x,y,z}$, then there seems to be no reason why the algebraic lattice $\idlat K$ should be completely generated by the set $\set{\ideal x, \ideal y, \ideal z}$ of principal ideals, because $a=\bigvee\set{b_i: i\in I}$ in $K$
need not imply that $\ideal a$ equals $    \bigvee\set{\ideal b_i: i\in I} =     \bigcup\set{\bigvee\set{\ideal b_j: j\in J}:J\subseteq I \text{ and }J\text{ is finite}}$.

\begin{proof}[{Proof of Corollary~\ref{coroltwo}}] By the Gr\"atzer--Schmidt theorem, see \cite{gratzerschmidt}, we can assume that our lattice $L$ is the congruence lattice of an algebra $A$. So $L$ is a complete sublattice of $\Equ(A)$. Since $|L|$ is accessible, the construction in \cite{gratzerschmidt} shows that $|A|$ is also accessible. (This is not surprising, since the class of sets with accessible cardinalities is closed under ``reasonable'' constructions.) Hence, Theorem~\ref{thmmain} yields that $L$ is a complete sublattice of a complete sublattice $K$ of some $\Equ(B)$ such that $K$ is completely generated by three elements. It is well known, see Gr\"atzer and Schmidt~\cite[Theorem 8]{gratzerschmidt} or Nation~\cite[Exercise 3.6]{nationbook},  that complete sublattices of algebraic lattices are algebraic. Therefore, since  $\Equ(B)$ is an algebraic lattice, so is $K$, as required.
\end{proof}

\subsection{Outline, prerequisite, and method} 
The rest of the paper is devoted to the proof of Theorem~\ref{thmmain}. 

Only  basic knowledge of lattice theory is assumed.  For example,  a small part of  each  of the books  Burris and Sankappanavar~\cite{burrissankappanavar},
Davey and Priestley~\cite{daveypriestley},
Gr\"atzer~\cite{gratzergltfound}, 
McKenzie, McNulty, and Taylor~\cite{mckenziemcnultytaylor}, and Nation~\cite{nationbook} is sufficient. 

Our approach towards Theorem~\ref{thmmain} has three main ingredients but they are mostly hidden behind the scenes. First, we need the fact that $\Equ(A)$ is completely generated by four elements if $|A|$ is accessible; see \cite{czgfourlarge} and \cite{czgooptwo} for the infinite case and  Strietz~\cite{strietz} and Z\'adori~\cite{zadori} for the finite case. Second,   \cite{czgtestlat} and  \cite{czgflomega} give the vague idea that we need some auxiliary graphs. Third, the appropriate graphs given in Figure~\ref{figmhhvsz} are taken from Skublics~\cite{skublics}. His graphs are symmetric;  this explains why they are more appropriate for our plan than those in \cite{czgtestlat}. (Actually, we have not checked whether the graphs from \cite{czgtestlat} could be used here.) Note that we cannot use the statements of  \cite{czgtestlat},  \cite{czgflomega}, and  \cite{skublics} in the present setting, because \cite{czgtestlat} and \cite{skublics} are related only to the  particular case  $|A_0|=2$ of  Lemma~\ref{keylemma}. Hence, we borrow only the ideas and some methods from these sources without explicit further reference.

\section{Auxiliary statements and the proof}
\subsection{Infinitary lattice terms} A cardinal $\carm$ is \emph{regular} if it is infinite and whenever $I$ is a set of cardinals smaller than $\carm$ and $|I|<\carm$, then $\sum\set{\lambda: \lambda\in I}<\carm$. 
For example, $\aleph_0$, $\aleph_1$, and $2^{\aleph_0}$ are regular. Let $\carm$ always denote a regular cardinal.  Following Gr\"atzer and Kelly~\cite{gratzerkelly}, a lattice $L$ is \emph{$\carm$-complete} if for every $\emptyset\neq X\subseteq L$ such that $X|<\carm$, both $\bigwedge X$ and $\bigvee X$ exists. Note that $\carm$ will always be chosen to be larger than the cardinality of the lattice we deal with, so $\carm$-completeness will simply mean the completeness of the lattice in question. We need $\carm$ only because we want to define \emph{$\carm$-terms} and their \emph{ranks} by transfinite induction as follows. 
Let $X$ be a set of variables. (Later, $X$ will consist of three or four elements.) 
Each element $x$ of $X$ is an $\carm$-term of rank $\rank(x)=0$. 
So the set of $\carm$-terms of rank 0 over $X$ is $T_0(\carm, X):=X$.
Assume that $\chi>0$ is an ordinal number and $T_\lambda(\carm,X)$ has already been defined for all ordinals $\lambda$ being less than $\chi$. Let $T'_\chi(\carm, X)$ be the set of all formal expressions $\bigvee(t_i:i \in I)$ and $\bigwedge(t_i:i \in I)$ where $2\leq |I|<\carm$ and for each $i\in I$, $t_i\in T_{\lambda_i}(\carm, X)$ for some $\lambda_i<\chi$. Note that $\bigvee(t_i:i \in I)$ is understood as the abbreviation of the pair $\tuple{\bigvee, \set{\pair{t_i}{i}:  i\in I}}$, and similarly for meets. (An alternative and more rigorous way, which  will not be followed, is to consider  $\bigvee(t_i: i \in I)$ a rooted tree whose leaves are labeled with variables, the rest of nodes with $\bigvee$ and $\bigwedge$, and the root is labeled by $\bigvee$.) We define 
\[
T_\chi(\carm,X):=T'_\chi(\carm,X)\cup  \bigcup\set{T_\lambda(\carm,X): \lambda<\chi}\text.
\]
Note that the sets $T_\chi(\carm,X)$, where $\chi$ ranges over the class of ordinals, form an increasing sequence. Finally, we let 
\[
T(\carm,X):=\bigcup\set{T_\chi(\carm,X):\chi\text{ is an ordinal}},
\]
which is the collection of \emph{$\,\carm$-terms} over $X$. 
For $t\in T(\carm,X)$, the \emph{rank} $\rank(t)$ of $t$ 
is the smallest ordinal $\chi$ with $t\in T_\chi(\carm,X)$. If $X=\set{x_1,\dots, x_n}$, then  $t\in T(\carm,X)$ is often denoted by $t(x_1,\dots, x_n)$ to indicate the set of variables.  
If $L$ is a complete lattice, $a_1,\dots,a_n$ are (not necessarily distinct) elements of $L$, and $t=t(x_1,\dots,x_n)\in T(\carm,\set{x_1,\dots,x_n})$, then the \emph{substitution} $t(a_1,\dots,a_n)\in L$ makes sense and it is defined by an obvious induction. The importance of $\carm$-terms is explained by the following lemma; its proof is a trivial induction and will be omitted.

\begin{lemma}\label{lemmateRmGen}
Let $L$ be a complete lattice, $|L|<\carm$, and let $a_1,\dots, a_n\in L$. Then  the smallest complete sublattice of $L$ that includes $\set{a_1,\dots, a_n}$ is
$\set{t(a_1,\dots, a_n): t\in T(\carm,\set{x_1,\dots,x_n})}$. 
In particular, $L$ is generated by $\set{a_1,\dots,a_n}$ as a complete lattice if and only if each $b\in L$ can be represented in the form $b=t(a_1,\dots,a_n)$ for some $t(x_1,\dots,x_n)\in T(\carm,\set{x_1,\dots,x_n})$. 
\end{lemma}

Note a particular case of Lemma~\ref{lemmateRmGen}: if $L$ is finite, then $\{t(a_1,\dots,a_n): t\in T(\aleph_0,\set{x_1,\dots,x_n})\}$ is the sublattice generated by $\set{a_1,\dots,a_n}$ in the usual sense.

\begin{figure}[htc]
\centerline
{\includegraphics[scale=1.0]{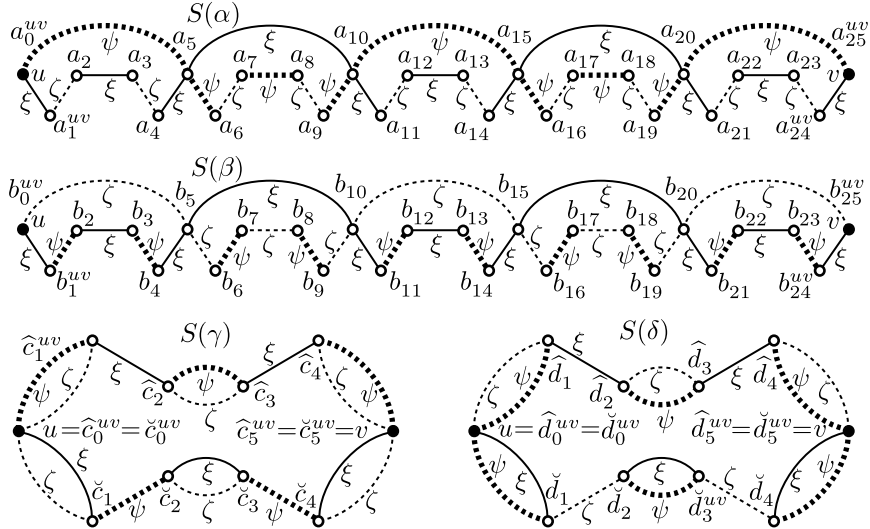}}
\caption{The auxiliary graphs\label{figmhhvsz}}
\end{figure}

\subsection{Blowing equivalence lattices up with auxiliary graphs}
The graphs $S(\alpha)$, \dots, $S(\delta)$ given in Figure~\ref{figmhhvsz} (but here we consider them without the superscripts $uv$) will be called  \emph{auxiliary graphs}; they are the key gadgets in our construction. Before formulating a lemma on these graphs, we need some easy 
definitions. Let $A_0$ be a set and let $\alpha_0,\beta_0,\gamma_0,\delta_0\in \Equ(A_0)$.   We define a larger set $A_1$ and equivalences $\xx_1,\yy_1,\zz_1\in \Equ(A_1)$ as follows. First,  we fix a wellordering on $A_0$ for convenience. A pair $\pair uv\in A_0^2$ is a \emph{nontrivial pair} if $u\neq v$. If $\pair uv\in A_0^2$ is a nontrivial pair such that  $u$ precedes $v$ with respect to the fixed wellordering, then  $\pair uv\in A_0^2$ is said to be an \emph{eligible pair}.
For each eligible pair $\pair uv\in \alpha_0$, we insert a copy of $S(\alpha)$ such that its black-filled left vertex is identified with $u$ while its black-filled right vertex with $v$. The vertices of  $S(\alpha)$ are $a_0,a_1,\dots, a_{25}$. 
After inserting a copy of this graph for $\pair uv$, its vertices are denoted by $\auv u v 0=u, \auv u v 1,\dots, \auv u v{24},\auv u v {25}=v$. The vertices $\auv u v 1,\dots, \auv u v{24}$ are new elements; they are neither in $A_0$, nor in any other copy of an auxiliary graph that we add to $A_0$.
Note that the superscript $uv$ is indicated at some vertices like $\auv u v 1$ in Figure~\ref{figmhhvsz} but, in absence of space, not always. If we drop the superscripts $uv$, then the figure gives $S(\alpha)$; if we add these superscripts, then we obtain the actual copy $\Suv u v(\alpha)$ of $S(\alpha)$ that we have inserted for $\pair uv$.
Similarly, for $\epsilon\in\set{\beta,\gamma,\delta}$ and each eligible pair $\pair uv\in\epsilon_0$, we insert a copy $\Suv u v(\epsilon)$  of $S(\epsilon)$ such that the left and right black-filled vertices are identified with $u$ and $v$, respectively. Again, the superscripts ensure that, apart possibly from the black-filled vertices, we insert disjoint copies. After all these insertions, we obtain $A_1$, which is a superset of $A_0$. 
The edges of our auxiliary graphs are colored by $\xx$, $\yy$, and $\zz$. These edges and their colors are also added to $A_0$. So $A_1$
 becomes a graph whose edges are colored by $\xx$, $\yy$, and $\zz$. 
This graph is denoted  by $A_1$, in the same way as its vertex set.  
Multiple edges between two vertices may occur, and each edge has a unique color in $\set{\xx,\yy,\zz}$. We define $\xx_1\in \Equ(A_1)$ by the rule $\pair xy\in\xx_1$ iff there is a path connecting $x$ and $y$ in the graph $A_1$ such that every edge of this path is $\xx$-colored. We define $\yy_1\in \Equ(A_1)$ and $\zz_1\in \Equ(A_1)$ analogously by using $\yy$-colored paths and $\zz$-colored paths, respectively. Consider the  ternary lattice terms 
\begin{equation}
\begin{aligned}
\hal&=\bigl(\xx \wedge(\yy \vee \zz )\bigr)\vee\bigl(\yy \wedge(\xx \vee \zz )\bigr),\cr
\hbe&=\bigl(\xx \wedge(\zz \vee \yy ))\vee\bigl(\zz \wedge(\xx \vee \yy )\bigr),\cr
\hga&=\bigl(\xx \vee(\yy \wedge \zz )\bigr)\wedge\bigl(\yy \vee(\xx \wedge \zz )\bigr),\text{ and}\cr
\hde&=\bigl(\xx \vee(\zz \wedge \yy )\bigr)\wedge\bigl(\zz \vee(\xx \wedge \yy )\bigr)\text.
\end{aligned}
\label{eqoUtmSs}
\end{equation}
Observe at this point that if we interchange $\yy$ and $\zz$ in our setting, then 
\begin{equation*}
\tuple{\hal,\hbe,\hga,\hde,S(\alpha), S(\beta),S(\gamma),S(\delta)}\,\text{ and }\,\tuple{\hbe,\hal,\hde,\hga,S(\beta), S(\alpha),S(\delta),S(\gamma)}
\end{equation*}
are also interchanged. We we will frequently rely on this fact, called \emph{$\yy$--$\zz$-symmetry}. Now, we are in the position to formulate the key lemma towards Theorem~\ref{thmmain}.

\begin{lemma}\label{keylemma} Let $A_0$ be a set with at least two elements and let 
$\alpha_0,\beta_0,\gamma_0,\delta_0\in \Equ(A_0)$ such that $\alpha_0\leq \gamma_0\vee \delta_0$ and  $\beta_0\leq \gamma_0\vee \delta_0$. Let $L_0$ be the complete sublattice of $\Equ(A_0)$ completely generated by $\set{\alpha_0,\beta_0,\gamma_0,\delta_0}$.
Consider the graph $A_1$ and the equivalences $\xx_1,\yy_1,\zz_1\in\Equ(A_1)$ defined above.  Denote by $L_1$ the complete sublattice of $Equ(A_1)$ completely generated by 
\[\{\hal(\xx_1,\yy_1,\zz_1), \hbe(\xx_1,\yy_1,\zz_1), \hga(\xx_1,\yy_1,\zz_1), \hde(\xx_1,\yy_1,\zz_1))\}\text.  
\] 
Then $L_0$ is isomorphic to a complete sublattice of $L_1$. 
\end{lemma}

\begin{proof} Let $\carm$ be a regular infinite cardinal such that  $\carm > 2^{|A_0|}$. Note that if $A_0$ is finite, then we choose $\carm=\aleph_0$. 
To ease the notation, for $\epsilon\in\set{\alpha,\beta,\gamma,\delta}$, we let $\hep_1=\hep(\xx_1,\yy_1,\zz_1)$. Then, by Lemma~\ref{lemmateRmGen}, 
\begin{equation}
L_1=\set{t(\hal_1,\hbe_1,\hga_1,\hde_1): t\in  T(\carm,\set{x_1,x_2,x_3,x_4})}\text.
\label{eqdefLonedkT}
\end{equation}

Let $\pair uv\in A_0^2$ be a nontrivial pair, that is, $u\neq v$. 
Let $\szi1=\xx_1 \wedge(\yy_1 \vee \zz_1)$. By a \emph{$(\yy\cup\zz)$-path} we mean a path $P$ in the graph $A_1$ such that each edge of $P$ is $\yy$-colored or $\zz$-colored. For a nontrivial pair $\pair xy\in A_1^2$,   
\begin{equation}
\parbox{10.5cm}{$\pair xy\in \xx_1$ iff there is a $\xx$-colored path connecting the vertices $x$ and $y$, and   $\pair xy\in \yy_1\vee\zz_1$ iff there is a $(\yy\cup\zz)$-path between  $x$ and $y$.}
\label{eqPthdSrfT}
\end{equation}
The \emph{restriction} of an equivalence $\mu$ to a subset $X$ will be denoted by $\restrict \mu X$. For $\mu\in\Equ(A_1)$ and $\epsilon\in\set{\alpha,\beta,\gamma,\delta}$, 
a list of all $\restrict\mu{S(\epsilon)}$-blocks will be called \emph{perfect} if for  every block $B$ in the list, $B\cap\set{u,v}=\emptyset$ implies that $B$ is also a $\mu$-block, not just a  $\restrict\mu{S(\epsilon)}$-block. From this aspects, all  singleton blocks will belong to the list even
if we give only the non-singleton blocks explicitly.  Perfectness will always be an obvious consequence of \eqref{eqPthdSrfT} or an analogous statement on other colors. It is easy to see that 
the $\restrict{(\yy_1\vee\zz_1)}{S(\alpha)}$-blocks are $\set{ \auv u v{1+5j},\auv u v{2+5j}}$ and $\set{ \auv u v{3+5j},\auv u v{4+5j}}$ for $j\in\set{0,2,4}$, and $\{u=\auv u v0$, $\auv u v5$, $\auv u v6$, $\auv u v7$, $\auv u v8$, $\auv u v9$, $\auv u v{10}$, $\auv u v{15}$, $\auv u v{16}$, $\auv u v{17}$,  $\auv u v{18}$, $\auv u v{19}$, $\auv u v{20}$,  $\auv u v{25}=v\}$; this list is perfect.
Hence,  
\begin{equation}
\parbox{8.0cm}
{the non-singleton $\restrict{\szi1}{\Suv u v(\alpha)}$-blocks  are  $
\set{\auv u v5,\auv u v{10}}$ and  $\set{\auv u v{15},\auv u v{20}}$, and this list is perfect.}
\label{eqszeblalph}
\end{equation}
Below, we obtain several similar observations easily; usually based on \eqref{eqPthdSrfT} (possibly for other colors) and, occasionally, on some earlier observations. Fortunately, the $\yy$--$\zz$-symmetry often offers a more economic argument; for example, this symmetry turns \eqref{eqszeblalph} into 
\begin{equation}
\parbox{8.0cm}
{The non-singleton $\restrict{\szi1}{\Suv u v(\beta)}$-blocks  are  $
\set{\buv u v5,\buv u v{10}}$ and  $\set{\buv u v{15},\buv u v{20}}$, and this list is perfect.}
\label{eqszepblBet}
\end{equation}
The $\restrict{(\yy_1\vee\zz_1)}{\Suv u v(\gamma)}$-blocks are  $\set{\cfel uv2,\cfel uv3}$ and $\Suv u v(\gamma)\setminus\set{\cfel uv2,\cfel uv3}$, and this list is perfect. 
This fact and $\yy$--$\zz$-symmetry yield that
\begin{equation}
\parbox{10 cm}
{The non-singleton $\restrict{\szi1}{\Suv u v(\gamma)}$-blocks  are $\set{u,\cdn uv1}$, $\set{\cdn uv2,\cdn uv3}$, and $\set{\cdn uv4,v}$, while the non-singleton $\restrict{\szi1}{\Suv u v(\delta)}$-blocks  are $\set{u,\ddn uv1}$, $\set{\ddn uv2,\ddn uv3}$, and $\set{\ddn uv4,v}$; both lists are perfect.}
\label{eqszijUverSt}
\end{equation}
Next, let $\szi2=\yy_1\wedge(\xx_1\vee \zz_1)$. The  $\restrict{(\xx_1\vee\zz_1)}{\Suv u v(\alpha)}$-blocks are $\{\auv u v {1+5j}$, $\auv u v {2+5j}\}$  and  $\{\auv u v {3+5j}$, $\auv u v {4+5j}\}$ for $j\in\set{1,3}$, and $\{u$, $\auv u v 1$, $\auv u v 2$, $\auv u v 3$, $\auv u v 4$, $\auv u v {5}$,   
$\auv u v {10}$, $\auv u v {11}$, $\auv u v {12}$,  $\auv u v {13}$, $\auv u v {14}$, $\auv u v {15}$, $\auv u v {20}$, $\auv u v {21}$, $\auv u v {22}$, $\auv u v {23}$, $\auv u v {24}$, $v\}$; this list is perfect. Thus,  
\begin{equation}
\parbox{8 cm}
{the non-singleton  $\restrict{\szi 2}{\Suv u v(\alpha)}$-blocks are $\set{u,\auv u v{5}}$,  $\set{\auv u v{10},\auv u v{15}}$, and $\set{\auv u v{20},v}$, and this list is perfect.}
\label{eqwqGbjTjN}
\end{equation}
The  $\restrict{(\xx_1\vee\zz_1)}{\Suv u v(\beta)}$-blocks are $\{\buv u v {2+5j}$, $\buv u v {3+5j}\}$ for $j\in\set{0,1,2,3,4}$ and  
$\{u$, $\buv u v 1$, $\buv u v 4$, $\buv u v 5$,  $\buv u v {6}$,  $\buv u v {9}$,   $\buv u v {10}$,  $\buv u v {11}$,  $\buv u v {14}$,  $\buv u v {15}$,  $\buv u v {16}$,  $\buv u v {19}$,  $\buv u v {20}$,  $\buv u v {21}$,  $\buv u v {24}$, $v\}$; this list is perfect. Thus,
\begin{equation}
\parbox{9 cm}
{all  $\restrict{\szi 2}{\Suv u v(\beta)}$-blocks are singletons, and this list is perfect.}
\label{eqwqjfgTmBxjpQN}
\end{equation}
The $\restrict{(\xx_1\vee\zz_1)}{\Suv u v(\gamma)}$-blocks  are 
 $\set{\cdn uv2,\cdn uv3}$ and $\Suv u v(\gamma)\setminus\set{\cdn uv2,\cdn uv3}$, and this list is perfect. On the other hand, $\Suv u v(\delta)$ is in itself a $\restrict{(\xx_1\vee\zz_1)}{\Suv u v(\delta)}$-block.  Hence, 
\begin{equation}
\parbox{11 cm}
{The non-singleton $\restrict{\szi2}{\Suv u v(\gamma)}$-blocks  are $\set{u,\cfel uv1}$, $\set{\cfel uv2,\cfel uv3}$, and $\{\cfel uv4$, $v\}$. The  $\restrict{\szi2}{\Suv u v(\delta)}$-blocks  are the same as the $\restrict{\psi_1}{\Suv u v(\delta)}$-blocks: 
 $\{u$, $\ddn uv1$, $\dfel uv1\}$, $\set{\ddn uv2,\ddn uv3}$, $\set{\dfel uv2,\dfel uv3}$, and   $\set{\ddn uv4,\dfel uv4, v}$. Both lists are perfect.}
\label{eqszijZtrSw}
\end{equation}
We know that $\hal_1=\szi1\vee\szi2$. Thus,  
\eqref{eqszeblalph}, 
\eqref{eqszepblBet}, \eqref{eqszijUverSt}, \eqref{eqwqGbjTjN}, \eqref{eqwqjfgTmBxjpQN}, and \eqref{eqszijZtrSw} imply that
the only non-singleton $\restrict{\hal_1}{\Suv u v(\alpha)}$-block is   $
\set{u,\auv u v{5},\auv u v{10},\auv u v{15},\auv u v{20},v}$; the non-singleton $\restrict{\hal_1}{\Suv u v(\beta)}$-blocks are
$\set{\buv u v5, \buv u v{10}}$ and $\set{\buv u v{15},\buv u v{20}}$; the non-singleton $\restrict{\hal_1}{\Suv u v(\gamma)}$-blocks are   $\set{u,\cdn uv1,\cfel uv1}$, $\set{\cdn uv2,\cdn uv3}$, $\{\cfel uv2$, $\cfel uv3\}$, and $\set{\cdn uv4,\cfel uv4,v}$;  and finally, the non-singleton $\restrict{\hal_1}{\Suv u v(\delta)}$-blocks are   $\{u,\ddn uv1$, $\dfel uv1\}$, $\set{\ddn uv2,\ddn uv3}$, $\set{\dfel uv2,\dfel uv3}$, and $\{\ddn uv4$, $\dfel uv4,v\}$. These lists are prefect; we will refer to them via Figure~\ref{figketto}, where the lists described above are indicated with appropriate edges. Note in advance that
\begin{equation}
\text{every list represented by Figure~\ref{figketto} is perfect.}
\label{eqfourlistPerF}
\end{equation}  
At present, we can assert 
 \eqref{eqfourlistPerF} and the validity of Figure~\ref{figketto} 
only for the lists represented by the thick dotted $\alpha_1$-edges, but there will be $\hbe_1$-edges, $\hga_1$-edges, and $\hde_1$-edges soon. 
The meaning of these  edges is that for $\epsilon,\mu \in\set{\alpha,\beta,\gamma,\delta}$ and $x,y\in\Suv u v(\mu)$, we have $\pair xy\in\restrict{\hep_1}{\Suv u v(\mu)}$ iff there is an $\hep_1$-colored path in $\Suv u v(\mu)$ connecting $x$ and $y$ in the figure. At present with $\epsilon=\alpha$, this means that if  $x,y\in\Suv u v(\mu)$, 
then  $\pair xy\in\restrict{\hal_1}{\Suv u v(\mu)}$ iff there is a thick dotted path in $\Suv u v(\mu)$ connecting $x$ and $y$ in the figure.
The $\yy$--$\zz$-symmetry takes care of the thick solid $\hbe_1$-edges.
Furthermore, again by this symmetry, it suffices to deal with the thin dotted $\hga_1$-edges. For $\epsilon\in\set{\alpha,\beta}$,
the intersection of any two of $\restrict{\xx_1}{\Suv u v(\epsilon)}$, $\restrict{\yy_1}{\Suv u v(\epsilon)}$, and $\restrict{\zz_1}{\Suv u v(\epsilon)}$ is the least equivalence on $\Suv u v(\epsilon)$. 
This makes it clear that for every $x$ in $\Suv u v(\epsilon)\setminus\set{u,v}$, the $( \restrict{\yy_1}{\Suv u v(\epsilon)}\wedge \restrict{\zz_1}{\Suv u v(\epsilon)})$-block of $x$ and the $(\restrict{\xx_1}{\Suv u v(\epsilon)} \wedge \restrict{\zz_1}{\Suv u v(\epsilon)} )$-block of $x$ are the singleton set $\set{x}$. Therefore, the $ \restrict{\hga_1}{\Suv u v(\epsilon)} $-block of $x$ is the  $(\restrict{\xx_1}{\Suv u v(\epsilon)}\wedge  \restrict{\yy_1}{\Suv u v(\epsilon)})$-block of $x$, which is $\set x$. Hence, the $\restrict{\hga_1}{\Suv u v(\epsilon)}$-blocks are singletons and this list is perfect. So the absence of thin dotted edges in $\Suv u v(\alpha)$ and $\Suv u v(\beta)$ in Figure~\ref{figketto} is correct.  
The $\restrict{(\xx_1\vee(\yy_1\wedge \zz_1))}{\Suv u v(\gamma)}$-blocks are $\set{\cdn uv2,\cdn uv3}$ and $\Suv u v(\gamma)\setminus \set{\cdn uv2,\cdn uv3}$, and this list is perfect. The $\restrict{(\yy_1\vee(\xx_1\wedge \zz_1))}{\Suv u v(\gamma)}$-blocks are $\set{\cfel uv2,\cfel uv3}$ and $\Suv u v(\gamma)\setminus \set{\cfel uv2,\cfel uv3}$, and this list is perfect again. Therefore, 
\begin{equation}
\parbox{8cm}
{the $\restrict{\hga_1}{\Suv u v(\gamma)}$-blocks are $\{u$, $v$, $\cfel uv1$, $\cdn uv1$, $\cfel uv4$, $\cdn uv4\}$, $\{\cfel uv2$, $\cfel uv3\}$, and $\{\cdn uv2$, $\cdn uv3\}$, and this list is perfect.}
\label{eqPrhgJmV}
\end{equation}
The $\restrict{(\xx_1\vee(\yy_1\wedge \zz_1))}{\Suv u v(\delta)}$-blocks are $\set{\ddn uv2,\ddn uv3}$ and $\Suv u v(\delta)\setminus \set{\ddn uv2,\ddn uv3}$. The $\restrict{(\yy_1\vee(\xx_1\wedge \zz_1))}{\Suv u v(\delta)}$-blocks are 
$\{\dfel uv2$, $\dfel uv3\}$, $\{\ddn uv2$, $\ddn uv3\}$, $\{u$, $\dfel uv1$, $\ddn uv 1\}$, and $\{v$, $\dfel uv4$, $\ddn uv 4\}$. Both lists are perfect. Hence, 
\begin{equation}
\parbox{8cm}
{the $\restrict{\hga_1}{\Suv u v(\delta)}$-blocks are $\{\dfel uv2$, $\dfel uv3\}$, $\{\ddn uv2$, $\ddn uv3\}$, $\{u$, $\dfel uv1$, $\ddn uv 1\}$, and $\{v$, $\dfel uv4$, $\ddn uv 4\}$, and this list is perfect.}
\label{eqPrhgJmV}
\end{equation}
This proves \eqref{eqfourlistPerF} and the correctness of Figure~\ref{figketto}. As an obvious consequence of Figure~\ref{figketto}, we note that for every eligible pair $\pair u v$,
\begin{equation}
\text{if }\epsilon\in\set{\alpha,\beta,\gamma,\delta}\text{ and }\pair u v\in\epsilon_0\text{, then }\pair u v\in\hep_1\text.
\label{eqHpeXtr}
\end{equation}

\begin{figure}[htc]
\centerline
{\includegraphics[scale=1.0]{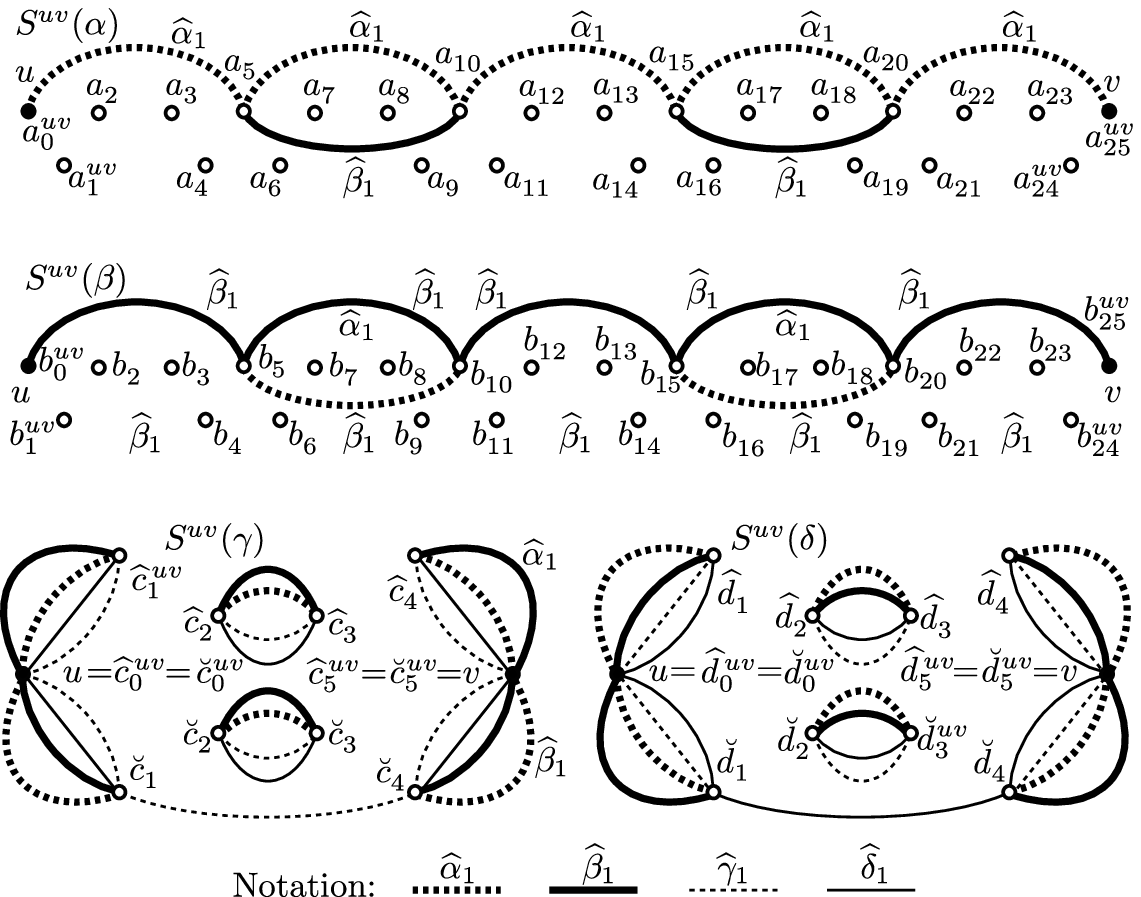}}
\caption{The restrictions of $\hal_1$, $\hbe_1$, $\hga_1$, and $\hde_1$ to the auxiliary  graphs\label{figketto}}
\end{figure}

Next, for $\epsilon\in\set{\alpha,\beta,\gamma,\delta}$, we let $\hep_2=\hep_1\wedge(\hga_1\vee\hde_1)$. Clearly $\hga_2=\hga_1$ and $\hde_2=\hde_1$. Similarly to \eqref{eqfourlistPerF} and Figure~\ref{figketto}, we have to explore how these equivalences are related to the auxiliary graphs. The situation will be summarized soon in 
Figure~\ref{figHarom}.
Note in advance that $\Suv u v(\gamma)$ and $\Suv u v(\delta)$ will be almost the same in 
Figure~\ref{figHarom} as in Figure~\ref{figketto}, whence only the difference will be indicated. 

First, consider an  eligible pair $\pair u v\in\alpha_0$. By the assumption $\alpha_0\leq \gamma_0\vee\delta_0$, there is a \emph{$(\gamma_0\cup\delta_0)$-sequence} $Q$ from $u$ to $v$ in $A_0$. This means that there is a finite sequence of elements of $A_0$ beginning with $u$ and ending with $v$ such that each pair of consecutive members of this sequence 
is collapsed by $\gamma_0$ or $\delta_0$. (As opposed to $A_1$, $A_0$ is not a graph. This is why  we do not call this sequence a  path.)
Applying \eqref{eqHpeXtr} to the pair $\pair u v$ and also to every edge of $Q$, we obtain that $\pair u v\in \hal_1\wedge(\hga_1\vee \hde_1)=\hal_2$. However, 
 since Figure~\ref{figketto} is perfect with respect to $\hga_1$ and $\hde_1$ by  \eqref{eqfourlistPerF}, all elements of $\Suv u v(\alpha)\setminus\set{u,v}$ are singleton blocks of $\hga_2=\hga_1$ and $\hde_2$, and thus they are singleton blocks of $\hal_2$.
Therefore, the situation for $\Suv u v(\alpha)$ is correctly depicted in the upper part of Figure~\ref{figHarom}, which gives a perfect list. 
We conclude similarly that Figure~\ref{figHarom} represents a perfect list for $\Suv u v(\beta)$. Since each $\alpha$-edge  within the graph $\Suv u v(\gamma)$ and 
$\Suv u v(\delta)$ is parallel to some $\gamma$-edge or $\delta$-edge, we have that  $\restrict{\hal_2}{\Suv u v(\gamma)}=\restrict{\hal_1}{\Suv u v(\gamma)}$ and $\restrict{\hal_2}{\Suv u v(\delta)}=\restrict{\hal_1}{\Suv u v(\delta)}$. We obtain in the same way, or from the ($\yy$--$\zz$)-symmetry, that 
$\restrict{\hbe_2}{\Suv u v(\gamma)}=\restrict{\hbe_1}{\Suv u v(\gamma)}$ and $\restrict{\hbe_2}{\Suv u v(\delta)}=\restrict{\hbe_1}{\Suv u v(\delta)}$.
Thus, since $\hep_2\subseteq \hep_1$ for all $\epsilon\in\set{\alpha,\beta, \gamma,\delta}$,  \eqref{eqfourlistPerF} yields that
\begin{equation}
\text{every list represented by  Figure~\ref{figHarom} is perfect.}
\label{eqfgHrpRfct}
\end{equation}

\begin{figure}[htc]
\centerline
{\includegraphics[scale=1.0]{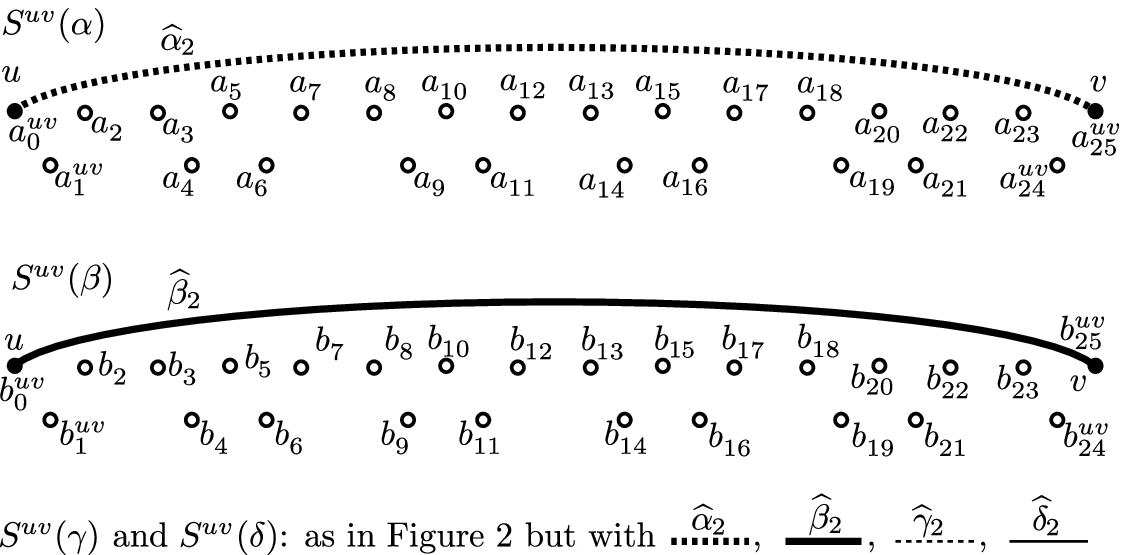}}
\caption{The restrictions of $\hal_2$, $\hbe_2$, $\hga_2$, and $\hde_2$ to the auxiliary graphs\label{figHarom}}
\end{figure}

The complete sublattice of $L_1$ (and also of  $\Equ(A_1)$) completely generated by $\set{\hal_2, \hbe_2, \hga_2, \hde_2}$ will be denoted by $L_2$.
Let $\Theta$ denote the smallest equivalence on $A_1$ collapsing each of the sets 
$\set{u,\cdn uv1,\cfel uv1}$, 
$\set{v,\cdn uv4,\cfel uv4}$, $\set{\cdn uv2,\cdn uv3}$, $\set{\cfel uv2,\cfel uv3}$ for eligible pairs $\pair uv\in \gamma_0$ and 
$\set{u,\ddn uv1,\dfel uv1}$, 
$\set{v,\ddn uv4,\dfel uv4}$, $\set{\ddn uv2,\ddn uv3}$, and $\set{\dfel uv2,\dfel uv3}$, for eligible $\pair uv\in \delta_0$. The two-element sets in the previous sentence  are $\Theta$-blocks by \eqref{eqfgHrpRfct} and Figures~\ref{figketto} and \ref{figHarom}, but the rest of these sets need not be such. However, 
\begin{equation}
\text{each $\Theta$-block contains a (unique) element of $A_0$.}
\label{eqTHeTablll}
\end{equation} 
Clearly, 
\begin{equation}
\Theta\leq \hep_2\text{ for all }\epsilon\in\set{\alpha,\beta,\gamma,\delta}\text{ and }\restrict {\Theta}{A_0}=\zero{A_0}\text.
\label{eqdgDJrw}
\end{equation}
Hence, by Lemma~\ref{lemmateRmGen}, $\Theta\leq \mu$ for all $\mu\in L_2$. 
Let  $L_3:=\set{\mu/\Theta: \mu\in L_2}\subseteq \Equ(A_1/\Theta)$. Using the folkloric  Correspondence Theorem from, say,  Burris and Sankappanavar~\cite[Theorem 6.20]{burrissankappanavar}, we obtain that $L_3$ is a complete sublattice of $\Equ(A_1/\Theta)$ and, furthermore, 
\begin{equation}
\parbox{9.5cm}{$f_1\colon L_2\to L_3$, defined by $\mu\mapsto \mu/\Theta$, is a lattice isomorphism, and
$L_3$ is generated by $\set{\hal_2/\Theta=f_1(\hal_2), \hbe_2/\Theta,\hga_2/\Theta,\hde_2/\Theta}$.}
\label{eqdghZMnNQb}
\end{equation}
Using the second half of \eqref{eqdgDJrw}, 
it follows that the map $A_0\to A_1/\Theta$, defined by $u\mapsto \blokk u\Theta$, is injective. Hence, the map 
\begin{equation}
\text{$g\colon A_0\to A_0/\Theta:=\set{\blokk u\Theta: u\in A_0}$, defined by $u\mapsto \blokk u\Theta$, is a bijection.}
\label{eqgTfnGjCtN}
\end{equation}
Since $A_0/\Theta$ is a subset of $A_1/\Theta$, the extension map $\Equ(A_0/\Theta)\to \Equ(A_1/\Theta)$, defined by $\mu\mapsto \mu\cup \zero{A_1/\Theta}$ is obviously a complete lattice embedding. Composing this map with the lattice isomorphism $\Equ(A_0)\to\Equ(A_0/\Theta)$ induced by $g$, we obtain that
the map 
\[f_2\colon \Equ(A_0)\to \Equ(A_1/\Theta)\text{, defined by }\mu\mapsto\set{\pair{\blokk u\Theta}{\blokk v\Theta}:\pair u v\in \mu}\cup \zero{A_1/\Theta},
\]
is a complete lattice embedding. Hence, 
\begin{equation}
\parbox{8cm}
{$L_4:=f_2(L_0)$ is isomorphic to $L_0$, and $L_4$ is completely generated by $\set{f_2(\alpha_0),f_2(\beta_0),f_2(\gamma_0),f_2(\delta_0) }$.}
\label{eqpnWErQTnbW}
\end{equation}
Next, we claim that for each $\epsilon\in\set{\alpha,\beta,\gamma,\delta}$,  $\hep_2/\Theta=f_2(\epsilon_0)$. To see this, 
let $\pair {\blokk u\Theta}{\blokk v\Theta}\in (A_1/\Theta)^2$ be a nontrivial pair. By \eqref{eqTHeTablll}, we can assume that $u$ and $v$ are in $A_0$.
If  $\pair {\blokk u\Theta}{\blokk v\Theta}\in f_2(\epsilon_0)$, then $\pair u v\in \epsilon_0$, Figure~\ref{figHarom} gives that $\pair u v\in\hep_2$, whence  $\pair {\blokk u\Theta}{\blokk v\Theta}\in \hep_2/\Theta$.  Conversely, assume that $\pair {\blokk u\Theta}{\blokk v\Theta}\in \hep_2/\Theta$. Hence, $\pair u v\in \hep_2$, and \eqref{eqfgHrpRfct} gives that $\pair u v\in \epsilon_0$. Thus, $\pair {\blokk u\Theta}{\blokk v\Theta}\in f_2(\epsilon_0)$, proving that  $\hep_2/\Theta=f_2(\epsilon_0)$.  
This equality, \eqref{eqdghZMnNQb}, and \eqref{eqpnWErQTnbW} yield that $L_3$ and $L_4$ have a common complete generating set. Thus, $L_4=L_3$. Hence, using  \eqref{eqdghZMnNQb} and \eqref{eqpnWErQTnbW} again, 
$L_2\cong L_0$. Therefore, since $L_2$ is a complete sublattice of $L_1$, the statement of Lemma~\ref{keylemma} follows.
\end{proof}

\subsection{The rest of the proof}
Now, armed with Lemma~\ref{keylemma} and some earlier results, we are in the position to complete the proof of Theorem~\ref{thmmain} in a short way.

\begin{proof}[Proof of Theorem~\ref{thmmain}]
Assume that $A$ is set with accessible cardinality at least 3. For convenience, we can also assume  that $A$ is not an even natural number, because otherwise we would construct a larger set $A'$ from $A$ by adding a new element and then we would use that $\Equ(A)$ is isomorphic to a sublattice of $\Equ(A')$. 

For a finite $A$, it was first proved by  Strietz~\cite{strietz} that $\Equ(A)$ is four-generated; however, we will use the four generators constructed by Z\'adori~\cite{zadori}. 
Note that these generators for $|A|=33$ are reproduced in 
\cite[Figure 1]{czgfourlarge}, and 33 is sufficiently large to indicate the general case for $|A|$ odd. Therefore, to be in harmony with the infinite case, we use 
$\fcitind\alpha{czgfourlarge}$, $\fcitind\beta{czgfourlarge}$, $\fcitind\gamma{czgfourlarge}$, and $\fcitind\delta{czgfourlarge}$ to denote Z\'adori's generators. The superscript ``f'' indicates that we are dealing with the finite case.    
Alternatively, we can use a different system of generators from  Z\'adori~\cite{zadori}, which are also given for $|A|=59$  in  \cite{czgooptwo}; these generators will be denoted by $\fcitind\alpha{czgooptwo}$, $\fcitind\beta{czgooptwo}$, $\fcitind\gamma{czgooptwo}$, and $\fcitind\delta{czgooptwo}$. Whichever of the two systems is considered, the join of two appropriately chosen generators is clearly $\unit{A}$. Namely, 
\begin{equation*}
\begin{aligned}
\fcitind\alpha{czgfourlarge}\vee \fcitind\beta{czgfourlarge}= \fcitind\alpha{czgfourlarge}\vee \fcitind\gamma{czgfourlarge} &=\fcitind\alpha{czgfourlarge}\vee \fcitind\delta{czgfourlarge}=\fcitind\beta{czgfourlarge}\vee \fcitind\gamma{czgfourlarge}=\unit{A}
\cr
 \text{and }
\fcitind\alpha{czgooptwo}\vee \fcitind\beta{czgooptwo} &= \fcitind\alpha{czgooptwo}\vee \fcitind\gamma{czgooptwo}= \fcitind\beta{czgooptwo}\vee \fcitind\gamma{czgooptwo}=\unit{A}\text.
\end{aligned}
\end{equation*}
Thus, we have many choices to fulfill the conditions $\alpha_0\leq\gamma_0\vee \delta_0$ and $\beta_0\leq\gamma_0\vee \delta_0$ of Lemma~\ref{keylemma}; let, say, $\tuple{\alpha_0,\beta_0,\gamma_0,\delta_0}= \tuple{\fcitind\gamma{czgfourlarge},\fcitind\delta{czgfourlarge},\fcitind\alpha{czgfourlarge},\fcitind\beta{czgfourlarge}}$. 

If $A$ is infinite but accessible, then we know from \cite{czgfourlarge} and \cite{czgooptwo} that $\Equ(A)$ is completely generated by four elements. Let $\{\icitind\alpha{czgfourlarge}$, $\icitind\beta{czgfourlarge}$, $\icitind\gamma{czgfourlarge}$, $\icitind\delta{czgfourlarge}\}$ and $\{\icitind\alpha{czgooptwo}$, $\icitind\beta{czgooptwo}$, $\icitind\gamma{czgooptwo}$, $\icitind\delta{czgooptwo}\}$ denote the complete generating sets constructed in \cite{czgfourlarge} and \cite{czgooptwo}, respectively. Here the superscript ``i'' comes from  ``infinite''. 
Again, we can pick two generators whose join is $\unit A$, but it suffices to see that one of the following six joins equals $\unit A$:
\begin{equation*}
\begin{aligned}
\icitind\alpha{czgfourlarge}\vee \icitind\beta{czgfourlarge} &= \icitind\alpha{czgfourlarge}\vee \icitind\gamma{czgfourlarge} =\icitind\alpha{czgfourlarge}\vee \icitind\delta{czgfourlarge}=\unit{A}\, \text{ and}
\cr
\icitind\alpha{czgooptwo}\vee \icitind\beta{czgooptwo} &= \icitind\alpha{czgooptwo}\vee \icitind\gamma{czgooptwo}= \icitind\beta{czgooptwo}\vee \icitind\gamma{czgooptwo}=\unit{A}\text.
\end{aligned}
\end{equation*}
Let, say, $\tuple{\alpha_0,\beta_0,\gamma_0,\delta_0}= \tuple{\icitind\gamma{czgfourlarge},\icitind\delta{czgfourlarge},\icitind\alpha{czgfourlarge},\icitind\beta{czgfourlarge}}$. 

Next, if $A$ is infinite, then let $A_0=A$, and consider the set $A_1$ constructed before Lemma~\ref{keylemma}. Clearly, $|A_1|=|A|$.
Denote by $K'$ the complete sublattice of $\Equ(A_1)$ completely generated by its three-element subset $\set{\xx_1,\yy_1,\zz_1}$. Since $L_0=\Equ(A_0)=\Equ(A)$ and $L_1$ from Lemma~\ref{keylemma} is clearly a complete sublattice of $K'$, Lemma~\ref{keylemma} gives that 
\begin{equation}
\text{$\Equ(A)$ is isomorphic to a complete sublattice of $K'$.}
\label{eqzQwvYtTnB}
\end{equation}
Now, let $B=A$. Since  $|A_1|=|A|=|B|$, we have an isomorphism $h\colon \Equ(A_1)\to \Equ(B)$. Clearly, $K:=h(K')$ is isomorphic to $K'$ and it is a complete sublattice of $\Equ(B)$. Also, $K$  is completely generated by the three-element set $\set{h(\xx_1),h(\yy_1),h(\zz_1)}$, and $\Equ(A)$ is isomorphic to a complete  sublattice of $K$. This proves the theorem for $A$ infinite.

If $A$ is finite, then we can drop ``complete'' from the consideration above. 
Let $B=A_1$; it is finite by construction. 
By \eqref{eqzQwvYtTnB}, $\Equ(A)$ is isomorphic to a sublattice of the three-generated $K'$ and $K'$ is a sublattice of $\Equ(B)$. Hence, we can let $K=K'$, which completes the proof of the theorem. 
\end{proof}

\end{document}